\documentclass[12pt, draft]{amsart}

\usepackage{enumerate}

\newtheorem{theorem}{Theorem}[section]
\newtheorem{proposition}[theorem]{Proposition}
\newtheorem{lemma}[theorem]{Lemma}
\newtheorem{corollary}[theorem]{Corollary}

\newtheorem{thm}{Theorem}

\newtheorem{mthm}{Main Theorem}

\theoremstyle{remark}
\newtheorem{remark}[theorem]{Remark}
\theoremstyle{definition}
\newtheorem*{definition}{Definition}
\theoremstyle{plain}

\allowdisplaybreaks

\begin{document}

\title[Zeros off the critical line for non-monomial polynomial of zeta-functions ]{On complex zeros off the critical line for non-monomial polynomial of zeta-functions}

\author{Takashi Nakamura}
\address{Department of Mathematics Faculty of Science and Technology \\Tokyo University of Science Noda, CHIBA 278-8510 JAPAN}
\email{nakamura\_takashi@ma.noda.tus.ac.jp}

\author{{\L}ukasz Pa\'nkowski}
\address{Faculty of Mathematics and Computer Science, Adam Mickiewicz University, Umultowska 87, 61-614 Pozna\'{n}, POLAND}
\email{lpan@amu.edu.pl}

\thanks{The first author was partially supported by JSPS Grants 21740024. The second author was partially supported by the grant no. N N201 6059 40 from National Science Centre.}



\begin{abstract}
In this paper, we show that any polynomial of zeta or $L$-functions with some conditions has infinitely many complex zeros off the critical line. This general result has abundant applications. By using the main result, we prove that the zeta-functions associated to symmetric matrices treated by Ibukiyama and Saito, certain spectral zeta-functions and the Euler-Zagier multiple zeta-functions have infinitely many complex zeros off the critical line. Moreover, we show that the Lindel\"of hypothesis for the Riemann zeta-function is equivalent to the Lindel\"of hypothesis for zeta-functions mentioned above despite of the existence of the zeros off the critical line. Next we prove that the Barnes multiple zeta-functions associated to rational or transcendental parameters have infinitely many zeros off the critical line. By using this fact, we show that the Shintani multiple zeta-functions have infinitely many complex zeros under some conditions. As corollaries, we show that the Mordell multiple zeta-functions, the Euler-Zagier-Hurwitz type of multiple zeta-functions and the Witten multiple zeta-functions have infinitely many complex zeros off the critical line.
\end{abstract}

\maketitle

\section{Introduction}
\subsection{Universality}
In 1975, Voronin \cite{Voronin} showed the universality theorem for the Riemann zeta-function $\zeta (s) := \sum_{n=1}^\infty n^{-s}$. To state it, let $D := \{ s \in {\mathbb{C}} : 1/2 < \Re (s) <1\}$ and $K\subset D$ be a compact set with connected complement. Denote by $\mu(A)$ the Lebesgue measure of the set $A$, and, for $T>0$, write $\nu_T \{ \cdots \} := T^{-1} \mu \{ \tau \in [0,T] : \cdots \}$ where the dots stand for a condition satisfied by $\tau$. Let $H_0(K)$ denote the space of continuous functions on $K$, which are analytic in the interior, equipped with the supremum norm $\Vert\cdot\Vert_K$ and $H(K)$ denote the subspace of $H_0(K)$ consisting of non-vanishing functions. Then the modern version of the Voronin theorem can be formulated as follows.
\begin{thm}\label{thm:Voronin}
For any $f\in H(K)$ and any $\varepsilon>0$, we have
\[
\liminf_{T\to\infty} \nu_T \bigl\{ \Vert \zeta(s+i\tau)-f(s) \Vert_K <\varepsilon \bigr\}>0.
\]
\end{thm}
Roughly speaking, this theorem implies that any non-vanishing analytic function can be uniformly approximated by the Riemann zeta-function. Subsequently, many mathematicians considered generalizations of universality (see for instance \cite{S}). For example, the strong universality property for the Hurwitz zeta-function $\zeta (s,a):= \sum_{n=0}^\infty (n+a)^{-s}$ was proved by Bagchi \cite{B} and Gonek \cite{Gonek}, independently (see \cite[Theorem 6.1.2 and Note in Section 6]{LauGa}). 
\begin{thm}\label{thm:2}
Let $0 < a < 1$ be transcendental. Then for any $f\in H_0(K)$ and any $\varepsilon > 0$, 
$$
\liminf_{T \rightarrow \infty} \nu_T \bigl\{ \Vert  \zeta (s+i\tau, a) - f(s) \Vert_K < \varepsilon \bigr\} > 0 .
$$
\end{thm}

As an application of strong universality for the Hurwitz zeta-function $\zeta (s,a)$, one can obtain the following theorem. The lower bound for the number of zeros of $\zeta (s,a)$ was first announced by Voronin \cite {Voronin} when $a$ is rational, but the proof (including transcendental $a$) was given by Bagchi \cite{B} and Gonek \cite{Gonek} (see also \cite[Theorems 8.4.7, 8.4.8 and 8.4.10, and Note in Section 8]{LauGa}). We write $f(x) \asymp x$ if there exist constants $0<C_1<C_2$ such that $C_1x \le f(x) \le C_2x$ for sufficiently large $x>0$. 
\begin{thm}\label{th:01}
For any $1/2 < \sigma_1 < \sigma_2 < 1$, the Hurwitz zeta-function $\zeta (s,a)$ where $a \ne 1/2, 1$ is rational or transcendental, has $\asymp T$ nontrivial zeros in the rectangle $\sigma_1 < \sigma < \sigma_2$, $0 < t <T$. 
\end{thm}

\subsection{Hybrid universality}
From our point of view, the most important result is so-called (strong) hybrid universality, which is a connection between the Voronin theorem and the classical Kronecker approximation theorem. Denote the distance to the nearest integer by $\Vert\cdot\Vert$. The precise definition is as follows.
\begin{definition}
\emph{Hybrid universality} for the function $L(s)$ is the following property: Let $K\subset D$, $f \in H(K)$ and $\{\alpha_j\}_{1\leq j \leq k}$ be real numbers linearly independent over $\mathbb{Q}$. Then for any $\varepsilon>0$ and any real numbers $\{\theta_j\}_{1\leq j \leq k}$, we have
\[
\liminf_{T\to\infty}\nu_T \bigl\{ 
\left\Vert L(s+i\tau) - f(s)\right\Vert_K < \varepsilon , \quad
\Vert\tau\alpha_j - \theta_j\Vert < \varepsilon, \,\,\, 1\leq j \leq k
\bigr\}>0.
\]
Moreover, we say that the function $L(s)$ is hybridly strongly universal if a function belonging to $H_0(K)$ can be approximated.
\end{definition}
The first result on hybrid universality was proved in weaker form by Gonek \cite{Gonek} and slightly improved using different method by Kaczorowski and Kulas \cite{KaczorowskiKulas}. They showed that Dirichlet $L$-functions $L(s,\chi) := \sum_{n=1}^\infty \chi (n) n^{-s}$ satisfy the inequality in above definition for $\alpha_n = \log{p_n}$, where $p_n$ denotes the $n$-th prime number. The second author \cite{PankowskiHybrid} proved the hybrid universality in the most general form for an axiomatically defined wide class of $L$-functions having Euler product, which contains for instance Dirichlet $L$-functions. Furthermore in \cite{PankowskiWithoutE} the hybrid strong joint universality theorem was proved for $L$-functions  without Euler product like Lerch zeta-functions, twists of Lerch zeta-functions and periodic Hurwitz zeta-functions.

As an application of the hybrid universality, we have the following theorem. Recall that a general Dirichlet series is an arbitrary series of the form 
\begin{equation}\label{eq:gds}
\sum_{n=1}^\infty a_n e^{-\lambda_n s}, \qquad a_n \in {\mathbb{C}}, \quad \lambda_n \in {\mathbb{R}} .
\end{equation}
The authors showed the following universality theorem for certain combinations of $L$-functions in \cite[Theorem 1]{NaPa1} by using hybrid universality. 

\begin{thm}\label{th:napamm1}
Suppose that $L(s)$ is hybridly universal, $A(s)$ and $B(s)$ are not identically vanishing general Dirichlet series, which are absolutely convergent in the half-plane $\sigma > 1/2$, and $A(s)$ is non-vanishing in $D$. Let $F(s) := A(s)L(s)+B(s)$, $K\subset D$ and $f(s)$ be continuous on $K$ and analytic in the interior. Moreover assume that $f(s) \ne B(s)$ for all $s \in K$. Then for any $\varepsilon>0$, it holds that
\[
\liminf_{T\to\infty}\nu_T \bigl\{ \Vert F(s+i\tau) - f(s)\Vert_K < \varepsilon \bigr\}>0.
\]
Furthermore, if the function $L(s)$ is hybridly strongly universal then the assumption that $f(s) \ne B(s)$ for all $s \in K$ can be relaxed. 
\end{thm}

Using the above theorem the authors \cite{NaPa1} obtained the universality theorem for zeta-functions associated to symmetric matrices treated by Ibukiyama and Saito \cite{IbuSa1} (see Theorem \ref{th:napamm2}). By improving the method used in \cite{NaPa1}, the authors \cite{NaPa2} proved the universality theorem for the Euler-Zagier multiple zeta-functions (see Theorem \ref{th:ezuni1}).

\subsection{Contents}
In the present paper, we give the lower and upper bound for the number of zeros of certain polynomials of $L$-functions in Main Theorem 1 and Main Theorem 2, respectively. Main Theorem 1 implies that any polynomial of hybrid universal function with some conditions has infinitely many complex zeros off the critical line. It should be emphasized that there are many zeta-functions whose all zeros lie on a line when they do not satisfy the assumption of Main Theorem 1 (see Remark \ref{rem:any}).  

By using the main theorems, we show that a lot of zeta or $L$-functions have infinitely many complex zeros off the critical line. In Theorem \ref{th:m2}, we prove that the zeta-functions associated to symmetric matrices (\ref{eq:ibusa1}) treated by Ibukiyama and Saito in \cite{IbuSa1} have infinitely many zeros in $n/2 < \Re (s) < (n+1)/2$. Next we show that certain spectral zeta-functions have infinitely many zeros in Theorem \ref{th:spe1}. We also prove in Theorem \ref{th:t1} that the Euler-Zagier multiple zeta-functions (\ref{eq:defez}) have zeros in $D$. Furthermore, we show in Propositions \ref{pro:ibpzli}, \ref{pro:spzli} and \ref{pro:ezli} that the Lindel\"of hypothesis for these zeta-functions are equivalent to the Lindel\"of hypothesis for the Riemann zeta-function. Hence we can say that the Lindel\"of hypothesis for zeta-functions mentioned above are independent of the Riemann hypothesis for these zeta-functions.

Moreover, we prove in Theorem \ref{th:zb1} that the Barnes multiple zeta-functions (\ref{eq:Bardef}) when $\lambda_1 =\cdots = \lambda_r =1$ and $a$ is not algebraic irrational have infinitely many zeros in the strip $r -1/2 < \Re (s) < r$. By using this fact, we show in Theorem \ref{th:sz1} that the Shintani multiple zeta-functions (\ref{eq:shi}) have infinitely many zeros in $r-1/2<\Re(s_1)<r$ and $0 \le \Re (s_k)$, $2 \le k \le m$ under some extra conditions. As corollaries, we show the same for the Mordell multiple zeta-functions (\ref{eq:mmzdef}), the Euler-Zagier-Hurwitz type of multiple zeta-functions (\ref{eq:ezhdef}), and the Witten multiple zeta-functions (\ref{eq:wmdef}).

\section{Main Theorems}
\subsection{Statement of main theorems}
In this section, we state the main theorems which give the lower and upper bound for the number of zeros of certain polynomials of $L$-functions. We should mention that Ka\v{c}inskait\.e, Steuding, \v{S}iau\v{c}i\=unas and R.~\v{S}le\v{z}evi\v{c}ien\.e (see \cite[Theorem 2]{Kach}) showed a special case of the following theorem, namely that $P(L(s))$ has infinitely many zeros to the right of the critical line $\Re (s)=1/2$ when $P \in {\mathbb{C}} [X]$ of degree $\geq 1$ is not a monomial.

Recall $D := \{ s \in {\mathbb{C}} : 1/2 < \Re (s) <1\}$ . Let us denote by $\mathcal{D}_s$ the ring of all general Dirichlet series defined as (\ref{eq:gds}) which are absolutely convergent in the half-plane $\Re (s) > 1/2$. Let $P_s \in \mathcal{D}_s[X]$ be a polynomial whose coefficients are in $\mathcal{D}_s$.
\begin{mthm}
Suppose that a function $L(s)$ is hybridly universal, $P_s \in \mathcal{D}_s[X]$ is not a monomial but a polynomial with degree greater than zero. Then, the function $P_s(L(s))$ has infinitely many zeros in $D$. More precisely, for any $1/2 < \sigma_1 < \sigma_2 < 1$, there exists a constant $C>0$ such that for sufficiently large $T$, the function $P_s(L(s))$  has more than $CT$ nontrivial zeros in the rectangle $\sigma_1 < \sigma < \sigma_2$, $0 < t <T$. Moreover, if the function $L(s)$ is hybridly strongly universal, the assumption that $P_s$ is not a monomial can be relaxed. 
\end{mthm}

Note that the cases, for example, $\zeta (s) + \zeta (2s)$, $\zeta^2 (s) - \zeta (s)$ are included but the case $\zeta (2s) \zeta (s)$ is excluded in the above theorem by the assumption $P_s$ is not a monomial. We obtain Theorem \ref{th:01} by the above theorem since $\zeta (s,a)$ where $a \ne 1/2, 1$ is rational or transcendental, is hybridly strongly universal. 
\begin{remark}\label{rem:any}
It should be emphasized that any zeta-function satisfying the assumption of Main Theorem 1 has infinitely many complex zeros in $D$. On the other hand, Taylor \cite{Taylor} showed that
$$
\zeta^* (s+1/2)-\zeta^* (s-1/2), \qquad \zeta^* (s) := \pi^{-s/2} \Gamma (s/2) \zeta (s)
$$
has all its zeros on the critical line $\Re (s)=1/2$. Thus there is a linear combination of the Riemann zeta-function multiplied by the Gamma function whose Riemann hypothesis holds. Namely, we can construct a zeta-function whose all zeros lie on the critical line if we relax the assumption of Main Theorem 1. Note that there are many other zeta-functions which have all their zeros on a line (see for example, Hejhal \cite{Hej}. Lagarias and Suzuki \cite{LaSu}, and Ki \cite{Ki}).
\end{remark}

\begin{mthm}
Suppose that $L(s)$ is a function defined as a Dirichlet series for $\sigma>1$ which can be continued analytically to a meromorphic function on $\Re (s)>1/2$ with a finite number of poles and all of them lie on the straight line $\Re (s)=1$. Moreover assume $L(s)$ is a function of finite order and for any fixed $1/2< \Re (s)<1$ the square mean-value 
\[
\int_0^T |L(\sigma+it)|^2dt \ll T \quad \mbox{as} \quad T\to\infty .
\]
Then, for any polynomial $P_s \in {\mathcal{D}}_s$ and any $\sigma_0>1/2$ we have
\[
N_{P(L)}(\sigma_0,T) \ll T,
\] 
where $N_{P(L)}(\sigma_0,T)$ counts the number of zeros $\rho$ with multiplicities of $P_s(L(s))$ with $\Re (\rho)>\sigma_0$ and $0<\Im(\rho)<T$.
\end{mthm}

\begin{remark}
Let $0 \ne a \in {\mathbb{C}}$. Then following the reasoning of the proof of Main Theorems 1 and 2 for the function $L(s)-a$, one can prove so-called $a$-values results of $L$-functions, where $L(s)$ satisfies the assumptions of Main Theorems (see \cite[Section 11]{Titch}).
\end{remark}

\subsection{Proof of Main theorem 1}
In order to prove Main Theorem 1, we quote the following theorem proved by the authors \cite[Corollary 2.2]{NaPa2}. 
\begin{thm}\label{th:napa2poly}
Suppose that a function $L(s)$ is hybridly universal and $P_s \in \mathcal{D}_s[X]$ is a polynomial with degree greater than zero. Then, for any $\varepsilon>0$ and any function $f(s) \in H(K)$, it holds that
\begin{equation}\label{eq:mtmlem}
\liminf_{T\to\infty}\nu_T \left\{ \left\Vert P_{s+i\tau} \bigl(L(s+i\tau) \bigr) - P_s \bigr(f(s)\bigr) \right\Vert_K < \varepsilon \right\}>0.
\end{equation}
Moreover if the function $L(s)$ is hybridly strongly universal, then $P_s \bigr(f(s)\bigr)$ for any function $f \in H_0(K)$ can be approximated.
\end{thm}
\begin{proof}
For the reader's convenience, we give a sketch of the proof. H.~Bohr in \cite{Bohr22} showed that for any Dirichlet series $\sum_{n=1}^\infty a_n n^{-s}$ and every $\varepsilon>0$ we have for an arbitrary $s$ lying in the half-plane of absolute convergence that
\[
\lim_{T\to\infty} \nu_T \left\{ \biggl|\sum_{n=1}^\infty\frac{a_n}{n^{s+i\tau}}-\sum_{n=1}^\infty\frac{a_n}{n^{s}}\biggr|<\varepsilon\right\}>0.
\]
In fact, Bohr applied the Kronecker theorem only for the sequence $\{\log p\}$, where $p$ runs over all primes. To obtain an analogous 
result for general Dirichlet series defined as (\ref{eq:gds}) it suffices to replace the sequence $\{\log p\}\subset\{\log n\}$ by the sequence $\{\lambda_{n_k}\}\subset\{\lambda_n\}$, where $\{\lambda_{n_k}\}$ is the basis of the vector space over $\mathbb{Q}$ generated by all $\lambda_n$. Then since hybrid universality combines diophantine approximations and the universality property, we obtain the theorem. 
\end{proof}
\begin{proof}[Proof of Main Theorem 1]
Firstly, suppose that the function $L(s)$ is hybridly universal. Let $z\in {\mathbb{C}}$ and fix a complex number $s_0 \in D$. Then $P_{s_0}(z)$ is a polynomial of variable $z$. Let $z_0 \in {\mathbb{C}}$ be a root of the equation $P_{s_0}(z)=0$. By the assumption $P_s$ is not a monomial, we can take $z_0 \ne 0$. Then we consider the function $f(s) := z_0 \exp(s-s_0)$ in Theorem \ref{th:napa2poly}. Because of the definition of $z_0$, one has $P_{s_0} (f(s_0))=0$. 

Let ${\mathcal{K}}$ be the closed disk whose center is $s_0$ and the radius is $r$. Moreover assume that $P_s (f(s))$ does not have zeros on the boundary of ${\mathcal{K}}$. Then by (\ref{eq:mtmlem}), we have 
\begin{equation}\label{eq:mtmpf1}
\max_{|s-s_0|=r} \bigl| P_{s+i\tau} \bigl(L(s+i\tau) \bigr) - P_s \bigr(f(s)\bigr) \bigr| <
\min_{|s-s_0|=r} \bigl| P_s \bigr(f(s)\bigr) \bigr| .
\end{equation}
An application of Rouch\'e's theorem shows that whenever the inequality (\ref{eq:mtmpf1}) holds, the function $P_{s+i\tau} (L(s+i\tau))$ has a zero in the interior of ${\mathcal{K}}$ (see also \cite[Section 8.1]{S}). According to (\ref{eq:mtmlem}), the measure of such $\tau \in [0,T]$ is greater than $CT$. This proves the theorem (see also \cite[Proof of Theorem 8.4.7]{LauGa}).

If $P_s$ is a monomial and the function $L(s)$ is hybridly strongly universal, it is sufficient to consider the function $s-s_0$ instead of $z_0 \exp(s-s_0)$. 
\end{proof}

\subsection{Proof of Main theorem 2}
We use the following two lemmas to show Main Theorem 2. The first one is called Littlewood's lemma (see for example \cite[lemma 7.2]{S}).
\begin{lemma}[Littlewood]
Let $g(s)$ be regular in and upon the boundary of the rectangle $\mathcal{R}$ with vertices $a$, $a+iT$, $b+iT$, $b$, and not zero on $\Re (s) =b$. Denote by $\nu(\sigma,T)$ the number of zeros $\rho$ inside the rectangle with $\Re (\rho) >\sigma$ including those with $\Im (\rho) =T$ but not $\Im (\rho) =0$. Then 
\[
\int_{\mathcal{R}}\log g(s)ds = -2\pi i\int_a^b\nu(\sigma,T)d\sigma.
\]
\end{lemma}
\begin{lemma}\label{lem:TendsToNoZero}
Assume that $f(s)$ is defined as (\ref{eq:gds}) for $\Re (s) >\sigma_0$. Then there 
exists $\theta\in\mathbb{R}$ such that $e^{\theta s}f(s)\to c\ne 0$ as $\Re(s)\to\infty$. Particularly, there exists $\sigma_f\geq\sigma_0$ such that $f(s)\ne 0$ for $\Re(s)>\sigma_f$.
\end{lemma}
\begin{proof}
Without loss of generality, we can assume that $a_n\ne 0$ and $\lambda_n\ne\lambda_m$ for $n\ne m$. 
Put $\lambda_{m_0}:=\min\{\lambda_m: m=0,1,2,\ldots\}$. Obviously $\lambda_{m_0} > -\infty$. Then, if $\lambda_{m_0}>0$, we easily see that $e^{\lambda_{m_0}s}f(s)\to a_{m_0}\ne 0$ as $\Re (s)\to\infty$. Otherwise, we have $e^{-\lambda_{m_0}s}f(s)\to a_{m_0}\ne 0$  as $\Re (s)\to\infty$.
\end{proof}

\begin{proof}[Proof of Main Theorem 2]
Without loss of generality, by Lemma~\ref{lem:TendsToNoZero}, we can assume that there is $\sigma_2>1$ such that $F(s) := P_s(L(s))\ne 0$ for $\Re (s) >\sigma_2$.

Let $1/2<\sigma_1<\sigma_0$, and $T_0>0$ be such that $F(\sigma+it)$ is regular for $|t|>T_0$ and $\sigma>1/2$. Then using Littlewood's lemma and comparing the imaginary parts gives
\begin{align*}
2\pi\int_{\sigma_1}^{\sigma_2} N_F(\sigma,T)d\sigma &= \int_{T_0}^T\log|F(\sigma_1+it)|dt -\int_{T_0}^T\log|F(\sigma_2+it)|dt\\
&\quad-\int_{\sigma_1}^{\sigma_2}\arg F(\sigma+i(T_0+T))d\sigma+\int_{\sigma_1}^{\sigma_2}\arg F(\sigma+iT_0)d\sigma\\
&=I_1+I_2+I_3+I_4,\qquad\text{say}.
\end{align*}

One can easily show that $I_2+I_4\ll T$ (see for example \cite[the proof of Theorem 7.1]{S}).

Let $m$ be the degree of $P_s(X)$. By Jensen's inequality and convexity of logarithm we get
\begin{align*}
I_1 &= \frac{m}{2}\int_{T_0}^T \log|F(\sigma+it)|^{2/m}dt \ll 
T\log \frac{1}{T-T_0}\int_{T_0}^T |F(\sigma+it)|^{2/m}dt \\
&\ll T \log \frac{1}{T} \sum_{n=0}^m \int_0^T|L(\sigma+it)|^{2n/m}dt \ll 
T\log\frac{1}{T}\int_0^T|L(\sigma+it)|^2dt\ll T.
\end{align*}

In order to evaluate $I_3$ let us observe that $|\arg F(\sigma+i(T_0+T))|\leq (N+1)\pi$, where $N$ denotes the number of zeros of $\Re (F(\sigma+i(T_0+T)))$ for $\sigma\in[\sigma_1,\sigma_2]$. To estimate $N$, put
\[
g(z)=\frac{1}{2} \Bigl( F(z+i(T_0+T)) + \overline{F(\overline{z}+i(T_0+T))} \Bigr),
\]
$R=\sigma_2-\sigma_1$ and $T$ larger than $2R$. Then it is easy to see that $g(z)$ is analytic for $|z-\sigma_2|<T$ and $N\leq n(R)$, where $n(r)$ denotes the number of zeros of $g(z)$ in $|z-\sigma_2|\leq r$. Then using Jensen's formula (see \cite[\S~3.61]{Titch}), the fact that $L(s)$ is a function of finite order and the following inequality
\[
\int_0^{2R}\frac{n(r)}{r}dr\geq n(R)\log 2
\]
gives the estimate $I_3\ll \log T$.

Collecting all estimates, we obtain
\[N_F(\sigma_0,T)\leq \frac{1}{\sigma_0-\sigma_1}\int_{\sigma_1}^{\sigma_0}N_F(\sigma,T)d\sigma \ll \int_{\sigma_1}^{\sigma_2}N_F(\sigma,T)\ll T,\]
which completes the proof.
\end{proof}

\section{Applications}

\subsection{Zeros of the zeta-functions associated to symmetric matrices}
Zeta-functions associated to the prehomogeneous vector space $V_n$ of $n \times n$ rational symmetric matrices are  interesting for several reasons. For example, their special values at non-positive integers describe the contribution of central unipotent elements to the dimensions of the Siegel modular form. For the case $n=2$, these zeta-functions have been investigated by Siegel, Morita, Shintani, Sato and Arakawa. For the case $n \ge 3$, explicit forms of the zeta-functions associated to symmetric matrices have been proved by Ibukiyama and Saito \cite{IbuSa1}. 

We use the same definitions and notation as in \cite{IbuSa1}. For each field $F$ in ${\mathbb{C}}$, we put $V_n (F) = V_n \otimes_{\mathbb{Q}} F$. Then the pair $(GL_n ({\mathbb{C}}), V_n ({\mathbb{C}}))$ is a prehomogeneous vector space through the action: $V_n ({\mathbb{C}}) \in x \longmapsto g x^t g \in V_n ({\mathbb{C}})$ for $g \in GL_n ({\mathbb{C}})$. Denote by $V_n^j$ the subset of  $V_n ({\mathbb{R}})$ consisting of matrices with $j$ positive and $n-j$ negative eigenvalues. We fix a lattice $L \subset V_n ({\mathbb{R}})$ which is invariant under $SL_n ({\mathbb{Z}})$ and put $L^{(j)} := L \cap V_n^j$ and denote by $L^{(j)}\!/\!\!\sim$ the set of $SL_n ({\mathbb{Z}})$ equivalence classes in $L^{(j)}$. Then, except for the case $n=2$ and $j=1$, the zeta-functions $\zeta_j (s,L)$ of signature $(j, n-j)$ attached to $L$ are given by the following series
$$
\zeta_j (s,L) := \frac{2 \prod_{k=1}^n \Gamma(k/2)}{\pi ^{n(n+1)/4}}
\sum_{x \in L^{(j)} / \sim} \frac{\mu (x)}{|\det x|^s}
$$
where $\mu (x)$ is defined as follows: let $G_+ := \{ g \in GL_n ({\mathbb{R}}) : \det g >0 \}$, and  $dg := (\det g)^{-n} \prod_{1 \le j,k \le n} dg_{jk}$ be the measure on $G_+$. As a measure on $V_n^j$, we take $dy := \prod_{1 \le j \le k \le n} dy_{jk}$. For $x \in L^{(j)}$, let $U$ be a relatively compact open set in $V_n^j$ and let $Y$ be the domain in $G_+$ which is mapped to $U$ by $x \longmapsto g x^t g$. Let $\Gamma_x$ be the stabilizer of $x$ in $SL_n ({\mathbb{Z}})$, and $Y_0$ a fundamental domain with respect to the right action of $\Gamma_x$. The ratio $\mu (x) := \int_{Y_0} dg / \int_U |y|^{-(n+1)/2} dy$ is finite and independent of the choice of $U$ unless $n=2$ and $j=1$. For a ring $R$, we denote by $S_n (R)$ the set of all symmetric matrices with coefficients in $R$ and by $S_n (R)_e$ its subset consisting of even elements, that is, the elements whose diagonal elements are contained in $2R$. By \cite[Lemma 1.1]{IbuSa1}, it is enough to consider the cases $L_n := S_n ({\mathbb{Z}})$ or $L_n^* := S_n ({\mathbb{Z}})_e/2$ when $n \ge 3$. Note that $\zeta_j (s,L)$, where $L=L_n$ or $L_n^*$ depends only on the determinant $\eta := (-1)^{n-j}$ and the Hasse invariant $\theta := (-1)^{(n-j)(n-j+1)/2}$ of $x \in V_n^j$ (see \cite[Section 2]{IbuSa1}). Hence we denote
$$
\zeta_{\eta, \theta} (s,L) := \zeta_j (s,L)
$$
for $L_n$ or $L_n^*$. For the zeta-functions associated to symmetric matrices $\zeta_{\eta, \theta}(s,L)$, the following theorem is proved by Ibukiyama and Saito \cite{IbuSa1}. For simplicity, we put
\begin{align*}\label{eq:bnBn}
b_n(s;L) &= \begin{cases}
|\prod_{k=1}^{[n/2]} B_{2k}| \Big/ 2^{n-1} (\frac{n-1}{2})! &\text{if $L=L_n$},\\
2^{(n-1)s} |\prod_{k=1}^{[n/2]} B_{2k}| \Big/ 2^{n-1} (\frac{n-1}{2})! &\text{if $L=L^*_n$},
\end{cases}\\
A_n(s;L) &= \begin{cases} 2^{(n-1)/2} \prod_{k=1}^{[n/2]} \zeta \bigl( 2s - (2k-1) \bigr) &\text{if $L=L_n$},\\
\prod_{k=1}^{[n/2]} \zeta \bigl( 2s - (2k-1) \bigr) &\text{if $L=L^*_n$},
\end{cases}\\
B_n (s) &= \theta \eta^{(n+1)/2} (-1)^{(n^2-1)/8} \zeta (s) \prod_{k=1}^{[n/2]} \zeta ( 2s - 2k ).
\end{align*}

\begin{thm}[{\cite[Theorem 1.2]{IbuSa1}}]\label{th:IbuSa1}
Let $n \ge 3$ be an odd positive integer. Then we have
\begin{equation}\label{eq:ibusa1}
\zeta_{\eta, \theta}(s,L) = b_n(s;L) \bigl(A_n(s;L)\zeta(s-(n-1)/2)+B_n(s)\bigr).
\end{equation}
\end{thm}
They also showed in \cite[Theorem 1.3]{IbuSa1} that when $n$ is even, these zeta-functions consist of two parts; one part is a product of shifted Riemann zeta-functions and the other part is composed of shifted Riemann zeta-functions and the Dirichlet series attached to an Eisenstein series of half integral weight of one variable. 

The universality theorem for $\zeta_{\eta, \theta}(s,L)$, where $n \ge 3$ is as follows. 
\begin{thm}[{\cite[Theorem 4]{NaPa1}}]\label{th:napamm2}
Assume that $L=L_n$ or $L^*_n$ and $n\ge 3$ is an odd positive integer. Let $K \subset D_n := \{s \in {\mathbb{C}} : n/2 < \Re (s) < (n+1)/2 \}$ be a compact set with connected complement. Moreover, let $f(s)$ be a function continuous on $K$ and analytic in the interior of $K$, such that $f(s)\ne b_n(s;L)B_n(s)$ for all $s\in K$. Then for any $\varepsilon > 0$, 
\begin{equation*}
\liminf_{T \rightarrow \infty} \nu_T \bigl\{ \Vert \zeta_{\eta, \theta} (s+ i\tau, L) - f(s) \Vert_K < \varepsilon \bigr\} > 0 .
\end{equation*}
\end{thm}

We have that $\zeta_{\eta, \theta} (s,L_n)$ and $\zeta_{\eta, \theta} (s,L_n^*)$, $n \in 2{\mathbb{N}}+1$ has zeros by the equation (\ref{eq:ibusa1}) and Main Theorems. It should be noted that zeta-functions associated to symmetric matrices have functional equations (see \cite[Introduction]{IbuSa1}). 
\begin{theorem}\label{th:m2}
Assume that $L=L_n$ or $L^*_n$ and $n\ge 3$ is an odd positive integer. Then for any $n/2 < \sigma_1 < \sigma_2 < (n+1)/2$, the zeta-function $\zeta_{\eta, \theta} (s,L)$ has $\asymp T$ nontrivial zeros in the rectangle $\sigma_1 < \sigma < \sigma_2$, $0 < t <T$. 
\end{theorem}

\subsection{Zeros of spectral zeta-functions}
Since the latter half of 1980's, the problem of evaluating the determinants of the Laplacians on Riemann manifolds has received considerable attention by many researchers including D'hoker and Phong, Sarnak, and Voros, who computed the determinants of the Laplacians on compact Riemann surfaces of constant curvature in terms of special values of the Selberg zeta-function. 

Let $\{\lambda_n\}$ be a sequence such that 
\begin{equation}\label{eq:lamdef}
0= \lambda_0 < \lambda_1 \le \lambda_2 \le \cdots \le \lambda_k \le \cdots, \qquad \lambda_k \to \infty, \quad 
k \to \infty; 
\end{equation}
henceforward we consider only such nonnegative increasing sequences. Then one can show that the spectral zeta function
$$
Z (s) := \sum_{k=1}^{\infty} \frac{1}{\lambda_k^{s}},
$$
which is known to converge absolutely in a half-plane $\Re (s) > \sigma_0$ for some $\sigma_0 \in {\mathbb{R}}$. Osgood, Phillips and Sarnak \cite{OPS} defined the determinant of the Laplacian $\Delta$ on the compact manifold $M$ by ${\rm{det}}' \Delta := \prod _{\lambda_k \ne 0} \lambda_k$, where $\{ \lambda_k \}$ is the sequence of eigenvalues of the the Laplacian $\Delta$ on $M$. The sequence $\{ \lambda_k \}$ is known to satisfy the condition (\ref{eq:lamdef}), but the product $\prod _{\lambda_k \ne 0} \lambda_k$ is always divergent. In order to make sense for the product, we must use some sort of regularization. It is easily seen that, formally, $e^{-Z'(0)}$ is the product of nonzero eigenvalues of $\Delta$. This product does not converge, but $Z(s)$ can continued analytically to a neighborhood of $s=0$. Thus we can give a meaningful definition
$$
{\rm{det}}' \Delta := \exp \bigl( - Z'(0) \bigr),
$$
which is called the Functional Determinant of the Laplacian $\Delta$ on $M$. 

We consider the sequence of eigenvalues on the standard Laplacian $\Delta_{{\mathbf{S}}^n}$ on the $n$-dimensional sphere ${\mathbf{S}}^n$ (see for example \cite[Chapter 8.3]{Tay}). The Laplacian $\Delta_{{\mathbf{S}}^n}$ has eigenvalues $\mu_k := k(k+n-1)$ with multiplicity
\begin{equation}\label{eq:muleigen}
\binom{k+n}{n} - \binom{k+n-2}{n} = \frac{(2k+n-1)(k+n-2)!}{k! (n-1)!}, \qquad k \in {\mathbb{N}}. 
\end{equation}
From now on we consider the shifted sequence $\{ \lambda_k \}$ of $\{ \mu_k \}$ by $(n-1)^2/4$ as a fundamental sequence. Then the sequence $\{ \lambda_k \}$ is written in the following tractable form:
\begin{equation}\label{eq:eigen2}
\lambda_k := \mu_k + \left( \frac{n-1}{2} \right)^2 = \left( k+ \frac{n-1}{2} \right)^2
\end{equation}
with the same multiplicity as in (\ref{eq:muleigen}). Hence their corresponding spectral zeta-functions $Z_{{\mathbf{S}}^n} (s)$ , $n= 1,2,3$ are as follows (see \cite[Chapter 5]{Sr}):
\begin{equation*}
\begin{split}
&Z_{{\mathbf{S}}^1} (s) = 2 \zeta (2s), \qquad Z_{{\mathbf{S}}^2} (s) = \bigl(2^{2s}-2\bigr)\zeta (2s-1) - 4^s, 
\qquad Z_{{\mathbf{S}}^3} (s) = \zeta (2s-2) -1,\\
&Z_{{\mathbf{S}}^4}(s) = \frac{1}{3} \bigl(2^{2s-3}-1\bigr) \zeta (2s-3) - 
\frac{1}{3} \biggl(2^{2s-3}-\frac{1}{4}\biggr) \zeta (2s-1) - \frac{1}{3} \biggl(\frac{2}{3}\biggr)^{2s-3} +
\frac{1}{8} \biggl(\frac{2}{3}\biggr)^{2s} .
\end{split}
\end{equation*}

By the definition of the sequence $\{ \lambda_k \}$ and their multiplicity, we have
\begin{equation}\label{eq:zns}
Z_{{\mathbf{S}}^n} (s) = A_{{\mathbf{S}}^n} (s) \zeta (2s-n+1) + B_{{\mathbf{S}}^n} (s) ,
\end{equation}
where $A_n (s)$ is a Dirichlet polynomial and $B_n(s)$ is a general Dirichlet series which is absolutely convergent in the half-plane $\sigma > n/2-1/4$ (see also the proof of Proposition \ref{pro:spzli}). Therefore we obtain the following theorem by Main Theorems. 
\begin{theorem}\label{th:spe1}
Let $n \ge 2$, Then for any $n/2-1/4 < \sigma_1 < \sigma_2 < n/2$ the spectral zeta-function $Z_{{\mathbf{S}}^n} (s)$ has $\asymp T$ nontrivial zeros in the rectangle $\sigma_1 < \sigma < \sigma_2$, $0 < t <T$. 
\end{theorem}

\subsection{Zeros of the Euler-Zagier multiple zeta-functions}
The Euler-Zagier multiple zeta-functions $\zeta_r (s_1, \ldots ,s_r)$ are defined by 
\begin{equation}\label{eq:defez}
\zeta_r (s_1, \ldots ,s_r) := \sum_{n_1> \cdots > n_r >0}^{\infty} \frac{1}{n_1^{s_1} \cdots n_r^{s_r}}, \qquad
\Re (s_1) >1, \quad \Re(s_j) \ge 1, \quad 2 \le j \le r.
\end{equation}
When $r=2$, Atkinson \cite{Atkinson} obtained an analytic continuation for $\zeta_2(s_1,s_2)$ in order to study the mean square $\int_0^T |\zeta (1/2+it)|^2dt$. He obtained an explicit formula for $\int_0^T |\zeta (1/2+it)|^2dt$ by using the harmonic product formula
$$
\zeta (s_1) \zeta (s_2) = \zeta_2(s_1,s_2) + \zeta_2(s_2,s_1) +\zeta (s_1+s_2) .
$$
More than fifty years later, Zhao \cite[Theorem 5]{Zhao}, and Akiyama, Egami and Tanigawa \cite[Theorem 1]{Aki}, independently, published the following analytic continuation for the Euler-Zagier multiple zeta-functions (see also a survey of Matsumoto \cite{Masmz}). 
\begin{thm}[{\cite[Theorem 5]{Zhao}}]
The Euler-Zagier multiple zeta-function $\zeta_r (s_1, \ldots ,s_r)$ can be analytically continued to a meromorphic function on all of ${\mathbb{C}}^r$ with possible poles at $s_1=1$ and $s_1 + \cdots + s_k = k-n$ for $1 \le k \le r$ and non-negative integer $n$. Moreover all the poles are simple.
\end{thm}
The values of the Euler-Zagier zeta-functions (after the continuation) at negative integer points are discussed by Akiyama and Tanigawa \cite{Aki2}. Moreover, Matsumoto showed a functional equation for the Euler-Zagier double zeta-functions in \cite{Ma2}. 

In the 1990s it was recognized that the values of the Euler-Zagier multiple zeta-function at integer points (multiple zeta-values) are quite important in the theory of quantum groups, knot theory, and so on. A lot of research on multiple zeta-values has then been done. For example, explicit expressions such as
$$
\zeta_r (2, \ldots ,2) = \frac{\pi^{2r}}{(2r+1)!}, \qquad \quad
\zeta_{2r} (3,1, \ldots, 3,1) = \frac{2\pi^{4r}}{(4r+2)!}
$$
have been shown. In recent years, many relations among multiple zeta-values were discovered by a lot of mathematicians, for instance, Ihara, Hoffman, Kaneko, Kawashima, Ohno and Zagier (see Hoffman's web page). 

Using the proof of Hoffman's relation \cite[Theorem 2.1]{Hoff}, one can show the following relation (see also \cite[Lemma 3.4]{NaPa2}). Let $\Pi$ denote a partition of the set $\{1,2,\ldots,r\}$. Moreover, for $\Pi=\{\varpi_1,\ldots,\varpi_l\}$, put
\[
c(\Pi) = (-1)^{r-l} \prod_{j=1}^l (|\varpi_j|-1)!\qquad\text{and}\qquad \zeta(s_1, \ldots ,s_r ; \Pi) = \prod_{j=1}^l\zeta \Biggl( \sum_{k\in \varpi_j}s_k \Biggr).
\]
Then, one can prove the following. For any $s_1,\ldots,s_r$ except for singularity,  we have
\begin{equation}\label{eq:ghof}
\sum_{\sigma\in\Sigma_r} \zeta_r (s_{\sigma(1)}, \ldots , s_{\sigma(r)}) = \sum_{{\rm{partitions}} \, \Pi \, {\rm{of}} \, \{1,\ldots,r\}} c(\Pi)\zeta(s_1, \ldots , s_r;\Pi),
\end{equation}
where $\Sigma_r$ denotes the symmetric group of degree $r$. Hence, there exists a polynomial $Z_r^{\#}(s)$ whose coefficients are Dirichlet series absolutely convergent in the half-plane $\Re (s) >1/2$, such that $\zeta_r(s,s,\ldots,s) = Z_r^{\#}(\zeta(s))$. For example, we have
$$
\zeta_2(s,s) = \frac{1}{2} \zeta (s)^2 - \frac{1}{2} \zeta (2s), \qquad
\zeta_3(s,s,s) = \frac{1}{6} \zeta (s)^3 - \frac{1}{2} \zeta (s) \zeta (2s) + \frac{1}{3} \zeta (3s). 
$$
By the relation (\ref{eq:ghof}), the authors showed the following theorem. 
\begin{thm}[{\cite[Theorem 3.5]{NaPa2}}]\label{th:ezuni1}
Let $g(s)$ be a function such that only $\zeta (s)$ is replaced by $f(s) \in H(K)$ in $\zeta (s ,\ldots, s ; \Pi_r)$ and $Z_r^{\#} (f(s),s) := \sum_{\Pi_r} c(\Pi_r) g(s)/r!$. Then for every $\varepsilon>0$, it holds that
\[
\liminf_{T\to\infty} \nu_T \left\{\Vert \zeta_r(s+i\tau, \ldots ,s+i\tau) - Z_r^{\#}(s) \Vert_K \right\}>0.
\]
\end{thm}

It should be noted that the first author showed the universality for the Euler-Zagier-Hurwitz type of multiple zeta-functions in \cite[Theorem 2.1]{Nakamura4}. Moreover he obtained relations between the zero-free region and the (joint) universality for Euler-Zagier-Hurwitz type of multiple zeta-functions (\ref{eq:ezhdef}). Afterwards, the authors showed not only a simple proof of it but also the more general cases in \cite[Theorems 3.2 and 3.3]{NaPa2} by the hybrid universality.

Zhao \cite[Section 5]{Zhao} obtained trivial zeros of the Euler-Zagier multiple zeta-functions. By analogy with the Riemann zeta-function, he propounded the following problem: ``Determine the complete set of trivial (resp. nontrivial) zeros of the multiple zeta-functions'' in \cite[Problem 1]{Zhao}. For nontrivial zeros, we obtain the following theorems and corollaries by the equation (\ref{eq:ghof}) and Main Theorems. 
\begin{theorem}\label{th:t1}
For any $1/2 < \sigma_1 < \sigma_2 < 1$, the Euler-Zagier multiple zeta-function $\zeta_r (s, \ldots ,s)$, $r \ge 2$ has $\asymp T$ nontrivial zeros in the rectangle $\sigma_1 < \sigma < \sigma_2$, $0 < t <T$. 
\end{theorem}
\begin{corollary}
The Euler-Zagier multiple zeta-function $\zeta_r (s_1, \ldots ,s_r)$ has zeros in $1/ 2 < \Re (s_1), \ldots , \Re (s_r) < 1$. 
\end{corollary}

\subsection{Lindel\"of hypothesis}
The Lindel\"of hypothesis for the Riemann zeta-function says that $\zeta (1/2 + it) = O(|t|^{\varepsilon})$. It is widely known that the Riemann hypothesis for $\zeta(s)$ implies the Lindel\"of hypothesis. Obviously, the Lindel\"of hypothesis for the function
$$
\zeta_{N} (s) := \zeta (s) - \sum_{n=1}^N n^{-s}
$$
is equivalent to the Lindel\"of hypothesis for $\zeta(s)$. Thus we can call $\Re (s) =1/2$ the critical line for $\zeta_N (s)$ since the estimation $\zeta_N (1/2 + it) = O(|t|^{\varepsilon})$ is equivalent to $\zeta (1/2 + it) = O(|t|^{\varepsilon})$. However, $\zeta_{N} (s)$ has infinitely many zeros off the critical line for $\zeta_N (s)$ by Main Theorems. More precisely, there exist a constant $C>0$ such that for sufficiently large $T$ the function $\zeta_N (s)$ has more than $CT$ nontrivial zeros in the rectangle $0 < t <T$, $\sigma_1 < \sigma < \sigma_2$  for any $1/2 < \sigma_1 < \sigma_2 < 1$. 

Next consider the Lindel\"of hypothesis for $\zeta_{\eta, \theta}(s,L)$ which states
$$
\zeta_{\eta, \theta}(n/2+it,L) = O(|t|^{\varepsilon}).
$$
By using (\ref{eq:ibusa1}), the estimation $\zeta (1+2it) = O(\log t)$ (see \cite[Theorem 5.16]{Tit}), and the fact that $b_n(s; L)$ and $A_n(s; L)$ are absolutely convergent general Dirichlet series for any odd $n \ge 3$, we have the following proposition. 
\begin{proposition}\label{pro:ibpzli}
The Lindel\"of hypothesis for $\zeta (s)$ is true if and only if the Lindel\"of hypothesis for $\zeta_{\eta, \theta}(s,L)$ is true for each $n \in 2{\mathbb{N}}+1$. 
\end{proposition}
\begin{proof}
It is well-know that
$$
0< \frac{\zeta (2\sigma)}{\zeta (\sigma)} = \prod_p \frac{1}{1+p^{-\sigma}} \le |\zeta(s)| \le \zeta (\sigma),
\qquad \sigma >1.
$$
Hence $b_n(s;L)$ and $A_n(s; L)$ are non-vanishing and bounded when $\Re (s)=n/2$. Moreover, we have $B_n(s; L)= O(\log t)$ from its definition, $\zeta (1+2it) = O(\log t)$ and the estimation above. Therefore we obtain this proposition by (\ref{eq:ibusa1}). 
\end{proof}

Similarly, we can define the Lindel\"of hypothesis for $Z_{{\mathbf{S}}^n}(s)$ which states
$$
Z_{{\mathbf{S}}^n} (n/2-1/4+it) = O(|t|^{\varepsilon}).
$$
We obtain the following proposition by (\ref{eq:zns}). 
\begin{proposition}\label{pro:spzli}
The Lindel\"of hypothesis for $\zeta (s)$ is true if and only if the Lindel\"of hypothesis for $Z_{{\mathbf{S}}^n}(s)$ is true for each $n \ge 2$. 
\end{proposition}
\begin{proof}
Suppose $n \ge 2$ and $\Re (s) > n/2$. Then we have
$$
Z_{{\mathbf{S}}^n} (s) = 
\sum_{k=1}^\infty \frac{(2k \!+\! n \!-\!1)(k \!+\! n \!-\! 2)!}{k! (n-1)!} 
\left( k+ \frac{n \!-\! 1}{2} \right)^{-2s} = 
A_{{\mathbf{S}}^n} (s)\zeta (2s \!-\! n \!+\! 1) + \cdots, 
$$
where $A_{{\mathbf{S}}^n} (s)$ is defined by 
$$
A_{{\mathbf{S}}^n} (s) := \frac{2}{(n-1)!} \begin{cases}
1 & n \mbox{ odd, }\\
2^{2s-n+1}-1 & n \mbox{ even,}
\end{cases}
$$
from (\ref{eq:muleigen}) and (\ref{eq:eigen2}). Obviously, $A_{{\mathbf{S}}^n} (s)$ does not vanish on the line $\Re (s) = n/2-1/4$. Therefore we obtain this proposition by (\ref{eq:zns}). 
\end{proof}

Finally, we can consider the Lindel\"of hypothesis for the Euler-Zagier multiple zeta-function $\zeta_r(s,\ldots ,s)$ which states that
$$
\zeta_r (1/2 + it, \ldots , 1/2 + it) = O(|t|^{\varepsilon}).
$$
Related to this problem, Ishikawa and Matsumoto \cite{IshiMa} proved an upper bound estimates for the Euler-Zagier multiple zeta-functions. Kiuchi and Tanigawa \cite{KiuTa} considered the problem of an order of magnitude for the triple zeta-functions of Euler-Zagier type in the region $0 < \Re (s_k) \le 1$, $1 \le k \le r$. It should be noted that Huxley \cite{Hux} showed $\zeta (1/2+it) = O(t^{H+\varepsilon})$, where $H:= 32/205=0.15609...$. Thus we have $$
\zeta_r (1/2 + it, \ldots , 1/2 + it) = O(|t|^{rH+\varepsilon})
$$ 
by the Hoffman formula (\ref{eq:ghof}). As an analogue of the zeta-functions associated to symmetric matrices, we have the following result. 
\begin{proposition}\label{pro:ezli}
The Lindel\"of hypothesis for $\zeta (s)$ is true if and only if the Lindel\"of hypothesis for $\zeta_r (s,\ldots ,s)$ is true for any $r \ge 2$. 
\end{proposition}
\begin{proof}
Suppose the Lindel\"of hypothesis for Riemann zeta-function $\zeta (s)$ is true. By the harmonic product formula, we have
$$
\zeta^2 (1/2+it) = 2\zeta_2 (1/2+it, 1/2+it) + \zeta (1+2it), \qquad t \ne 0 .
$$
Because of $\zeta (1+2it) = O(\log t)$, the Lindel\"of hypothesis for $\zeta_2 (s,s)$ is true. Obviously, the opposite is also true. Using the Hoffman formula (\ref{eq:ghof}), we have
$$
\zeta_3 (1/2+it,1/2+it,1/2+it) = \zeta^3 (1/2+it)/6 - \zeta (1/2+it)\zeta(1+2it)/2 + \zeta (3/2+3it), 
\quad t \ne 0 .
$$
Thus we obtain the case $r=3$. Similarly, by the Hoffman formula (\ref{eq:ghof}), we have
\begin{equation*}
\begin{split}
&\zeta_r (1/2 + it, \ldots , 1/2 + it) = \zeta^r (1/2+it)/r! \\
&+ O\bigl( \zeta^{r-2} (1/2+it) \log t \bigr) + O\bigl( \zeta^{r-3} (1/2+it) \log t \bigr) + \cdots + O\bigl( \log ^r t \bigr) .
\end{split}
\end{equation*}
This equality implies Proposition \ref{pro:ezli}. 
\end{proof}

\begin{remark}
The Lindel\"of hypothesis for $\zeta_N (s)$, $\zeta_{\eta, \theta}(s,L)$, $n \in 2{\mathbb{N}}+1$, $Z_{{\mathbf{S}}^n}(s)$, $n \ge 2$ or $\zeta_r(s,\ldots ,s)$, $r \ge 2$ is equivalent to the Lindel\"of hypothesis for $\zeta (s)$. However, these functions have infinitely many zeros off the line $\Re (s) = 1/2$, $\Re (s) = n/2$, $\Re (s) = n/2-1/4$, or $\Re (s) = 1/2$, respectively. Hence we can say that the Lindel\"of hypothesis for $\zeta_{\eta, \theta}(s,L)$, $n \in 2{\mathbb{N}}+1$, $Z_{{\mathbf{S}}^n}(s)$, $n \ge 2$ and $\zeta_r(s,\ldots ,s)$ , $r \ge 2$ are independent of the Riemann hypothesis for these zeta-functions.
\end{remark}

\subsection{Zeros of the Barnes and Shintani multiple zeta-functions}
Barnes \cite{Ba} considered the multiple sum of the form
\begin{equation}\label{eq:Bardef}
\zeta_r (s, a \,|\, \Lambda) := \sum_{n_1 ,\ldots, n_r =0}^{\infty} (\lambda_1n_1+\cdots+\lambda_rn_r+a)^{-s},
\qquad \Re (s) >r \ge 2,
\end{equation}
where $a$, $\lambda_1, \ldots \lambda_r$ are complex numbers satisfy some conditions. Nowadays this is called the Barnes $r$-tuple zeta-function. Barnes proved that the function $\zeta_r (s \,; a \,|\, \Lambda)$ can be continued meromorphically to the whole $s$-plane and is holomorphic except for simple poles at $s= 1, \ldots, r$. Barnes defined the multiple gamma function by $\zeta_r (s, a \,|\, \Lambda)$ and studied its properties. Afterwards many mathematicians have studied properties of the Barnes multiple zeta-functions (see for example \cite[Section 1]{Masmz}). 

In this paper, we concern only the simple case $0<a $ and $\lambda_1 =\cdots = \lambda_r =1$. For simplicity, in this case we write $\zeta_r (s, a)$ instead of $\zeta_r (s, a \,|\, \Lambda)$. In \cite[p.~86]{Sr}, it was showed that 
\begin{equation}\label{eq:sri1}
\zeta_r (s,a) = \sum_{j=0}^{r-1} p_{rj} (a) \zeta (s-j,a) , \quad 
p_{rj} (a) := \frac{1}{(r-1)!} \sum_{l=j}^{r-1} (-1)^{r+1-j} \binom{l}{j} s(r,l+1) a^{l-j},  
\end{equation}
where $s(r,l+1)$ is the Stirling number of the first kind. For example, one has
\begin{equation*}
\begin{split}
&\zeta_2 (s,a) = (1-a) \zeta (s,a) + \zeta(s-1,a), \\
&\zeta_3 (s,a) = (a^2-3a+2) \zeta (s,a)/2 + (3-2a)\zeta(s-1,a)/2 + \zeta (s-2,a)/2 .
\end{split}
\end{equation*}
Recall that $\zeta (s,a)$ with $a=1,1/2$ is hybridly universal, and $\zeta (s,a)$ with rational $a \ne 1,1/2$ or transcendental $a$ is strongly hybridly universal (see Section 1). Therefore, on zeros for $\zeta_r (s, a)$ with $r \ge 2$, we have the following theorems by the equation (\ref{eq:sri1}), Main Theorems and the fact mentioned above.
\begin{theorem}\label{th:zb1}
Suppose that $a$ is not algebraic irrational. Then for any $r-1/2 < \sigma_1 < \sigma_2 < r$, the function $\zeta_r (s, a)$ with $r \ge 2$ has $\asymp T$ nontrivial zeros in the rectangle $\sigma_1 < \sigma < \sigma_2$, $0 < t <T$. 
\end{theorem}

Next we consider the Shintani multiple zeta-functions defined by 
\begin{equation}\label{eq:shi}
\zeta_r ({\mathbf{s}}, {\mathbf{a}} \,|\, \Lambda) := \sum_{n_1 ,\ldots, n_r =0}^{\infty} \prod_{l=1}^m 
\bigl(\lambda_{l1}(n_1+a_1) +\cdots+\lambda_{lr}(n_r+a_r) \bigr)^{-s_l},
\end{equation}
where $a_k, \lambda_{lk} >0$, $1 \le l \le m$ and $1 \le k \le r$. Imai \cite{Imai} and Hida \cite{Hida} considered the above series or more generally with characters in the numerator. In \cite[Lemma 2.4.1]{Hida}, it was showed that $\zeta_r ({\mathbf{s}}, {\mathbf{a}} \,|\, \Lambda)$ converges absolutely and uniformly on any compact subset in the region $\Re (s_l) > r/m$ for all $1 \le l \le m$. Moreover, in \cite[Theorem 2.4.1]{Hida}, it was proved that $\zeta_r ({\mathbf{s}}, {\mathbf{a}} \,|\, \Lambda)$ can be continued to the whole space ${\mathbb{C}}^m$ as a meromorphic function. 

Obviously, this is a generalization of the Barnes multiple zeta-function defined by (\ref{eq:Bardef}). Shintani (see for example \cite{Shi}) introduced the above series with $s_1 = \cdots =s_m$ in order to research the Barnes multiple gamma function. Cassou-Nogu\`es (see for example \cite{Ca}) inspired by Shintani's work, considered those multiple series of the form that the numerator of (\ref{eq:shi}) is multiplied by certain roots of unity with $s_1 = \cdots =s_m$. She proved its meromorphic continuation and gave applications to $L$-functions and $p$-adic $L$-functions of totally real number fields (see a survey \cite[Section 2]{Masmz}). 

By Theorem \ref{th:zb1}, we obtain the following theorem. 
\begin{theorem}\label{th:sz1}
Suppose that $\lambda_{11} = \cdots = \lambda_{1r} =1$, and $a_1+\cdots + a_r$ is not algebraic irrational. Then $\zeta_r ({\mathbf{s}}, {\mathbf{a}} \,|\, \Lambda) $ has zeros in $r-1/2<\Re(s_1)<r$ and $0 \le \Re (s_k)$, $2 \le k \le m$. 
\end{theorem}
The Shintani multiple zeta-functions contain many important multiple zeta-functions. For instance, the Mordell multiple zeta-functions, the Euler-Zagier-Hurwitz type of multiple zeta-functions, and the Witten multiple zeta-functions are contained.

For example, the Mordell multiple zeta-functions are defined by 
\begin{equation}\label{eq:mmzdef}
\sum_{n_1, \ldots , n_r =1}^{\infty} n_1^{-s_1} \cdots n_r^{-s_r} (n_1+ \cdots + n_r+a)^{-s_{r+1}}, 
\qquad a \ge 0
\end{equation}
Tornheim \cite{Tor} considered the values at positive integer when $a=0$ and $r=2$. Mordell \cite{Mor} studied the case $a=0$, $r=2$ and $s_1 = s_2 = s_3$, and also studied the case $s_1 = \cdots = s_{r+1}$. By Theorem \ref{th:sz1}, we can see that the Mordell multiple zeta-functions have infinitely many zeros in $r-1/2<\Re(s_{r+1})<r$ and $0 \le \Re (s_j)$, $1 \le j \le r$, unless $a$ is algebraic irrational. 

Next, let us introduce the basic properties of the Euler-Zagier-Hurwitz type of multiple zeta-functions defined by
\begin{equation}\label{eq:ezhdef}
\begin{split}
&\zeta_r (s_{(1, \ldots , r)}\,; a_{(1, \ldots , r)}) := \sum_{n_1 > n_2 > \cdots > n_r \ge 0} 
\frac{1}{(n_1+a_1)^{s_1} (n_2+a_2)^{s_2} \cdots (n_r+a_r)^{s_r}} \\
&=\sum_{\substack{n_1, \ldots ,n_{r-1}=1 \\ n_r =0}}^\infty
\frac{1}{(n_1+\cdots + n_r + a_1)^{s_1} (n_2+\cdots + n_r +a_2)^{s_2} \cdots (n_r+a_r)^{s_r}} ,
\end{split}
\end{equation}
where $s_{(1, \ldots , r)} := (s_1 ,\ldots, s_r) \in {\mathbb{C}}^r$ and $a_{(1, \ldots , r)} := (a_1 ,\ldots , a_r) \in (0,1]^r$. This series absolutely converges for $\Re (s_1) > 1$ and $\Re (s_j) \ge 1$, $2\le j \le r$. Obviously $\zeta_r (s_{(1, \ldots , r)}\,; a_{(1, \ldots , r)})$ is a generalization of the Euler-Zagier multiple zeta-functions (\ref{eq:defez}). The meromorphic continuation of $\zeta_r (s_{(1, \ldots , r)}\,; a_{(1, \ldots , r)})$ to ${\mathbb{C}}^r$ has been accomplished by Akiyama and Ishikawa \cite{AI} using the Euler-Maclaurin summation formula, Matsumoto \cite{MaN} using the Mellin-Barnes formula, and by Murty and Sinha \cite{MS} using the binomial theorem and Hartog's theorem. Moreover, Kelliher and Masri \cite{KeMa} not only proved a meromorphic continuation but also provided a way to locate trivial zeros of $\zeta_r (s_{(1, \ldots , r)}\,; a_{(1, \ldots , r)})$. The first author \cite{Nakamura4} showed the existence of the zeros for $\zeta_r (s_{(1, \ldots , r)}\,; a_{(1, \ldots , r)})$ in the region of absolute convergence when $a_1 ,\ldots , a_r$ are algebraically independent. He also showed $\zeta_r (s_{(1, \ldots , r)}\,; a_{(1, \ldots , r)})$ has infinitely many zeros in $1/2 < \Re (s_1) <1$ when $a_1$ is transcendental and $\zeta_{r-1} (s_{(2, \ldots , r)} \,; \alpha_{(2, \ldots , r)} ) \ne 0$ for fixed $\Re (s_2) > 3/2$ and $\Re (s_l) \ge 1$, $3 \le l \le r$. By Theorem \ref{th:sz1}, we can see that the Euler-Zagier-Hurwitz type of multiple zeta-functions have infinitely many zeros in $r-1/2<\Re(s_1)<r$ and $0 \le \Re (s_j)$, $2 \le j \le r$, unless $a_1$ is algebraic irrational. 

At the end, let us define the Witten multiple zeta-functions. Let ${\mathfrak{g}}$ be a complex semisimple Lie algebra, and define
$$
\zeta (s \,; {\mathfrak{g}}) := \sum_{\smash{\rho}} \bigl(\dim \rho \bigr)^{-s},
$$
where $\rho$ runs over all finite dimensional irreducible representations of a certain semisimple Lie algebra ${\mathfrak{g}}$. Special values of this Dirichlet series were firstly studied by Witten \cite{Witten} in connection with quantum gauge theory. And Zagier \cite{Zaecm} called the above function as the Witten zeta-function associated with ${\mathfrak{g}}$. Some evaluation formulas for $\zeta (s \,; {\mathfrak{g}})$ at positive even integral integers were given by Zagier \cite{Zaecm}, Gunnells and Sczech \cite{Gu}. In order to analyze this multiple series closely, Komori, Matsumoto and Tsumura \cite{Ko2} introduced the following multi-variable version of the series. 

Let $r$ be the rank of ${\mathfrak{g}}$. Denote by $\Delta = \Delta({\mathfrak{g}})$ the set of all roots of ${\mathfrak{g}}$, by $\Delta_+ = \Delta_+({\mathfrak{g}})$ the set of all positive roots of ${\mathfrak{g}}$. For any $\alpha \in \Delta$, we denote by $\alpha^{\!\lor}$ the associated coroot. Let $\lambda_1, \ldots , \lambda_r$ be the fundamental weights satisfying $\langle \alpha_j^{\!\lor}, \lambda_k \rangle = \delta_{jk}$ (Kronecker's delta). Define the Witten multiple zeta-functions associated with semisimple Lie algebras by
\begin{equation}\label{eq:wmdef}
\zeta_r ({\mathbf{s}} \,; {\mathfrak{g}}) := \sum_{n_1, \ldots , n_r =1}^{\infty} \prod_{\alpha \in \Delta_+}
\langle \alpha^{\!\lor}, n_r \lambda_1 + \cdots + n_r \lambda_r \rangle^{-s_\alpha},
\end{equation}
where ${\mathbf{s}} := (s_{\alpha})_{\alpha \in \Delta_+} \in {\mathbb{C}}^m$ (here $m=|\Delta_+|$ is the number of positive roots of ${\mathfrak{g}}$). When ${\mathfrak{g}}$ is type of $X_r$, where $X=A,B,C,D,E,F,G$, Komori, Matsumoto and Tsumura \cite{Ko2} gave explicit formulas of (\ref{eq:wmdef}). For example (see \cite[(2.3)]{Ko2}), we have 
$$
\zeta_r ({\mathbf{s}} \,; A_r) = \sum_{n_1, \ldots , n_r =1}^{\infty} \prod_{1 \le j < k \le r+1}
(n_k+ \cdots + n_{j-1})^{-s_{jk}} , \qquad {\mathbf{s}} := (s_{jk}) \in {\mathbb{C}}^{r(r+1)/2} .
$$
Hence, \cite[Theorem 5.4]{Ko2} and Theorem \ref{th:sz1} imply that the Witten multiple zeta-functions also have infinitely many zeros when ${\mathfrak{g}} = A_r$ with $r \ge 2$ (by \cite[(2.3)]{Ko2}), $B_r$ with $r \ge 2$ (by \cite[(2.9)]{Ko2}), $C_r$ with $r \ge 2$ (by \cite[(2.15)]{Ko2}), $D_r$ with $r \ge 3$ (by $D_2 \simeq A_1 \oplus A_1$ and \cite[(2.21)]{Ko2}), $E_6, E_7, E_8, F_4, G_2$ (by \cite[diagrams in p.~383 and 384]{Ko2}). 


\end{document}